\documentclass[preprint]{elsarticle}

\usepackage{amssymb}
\usepackage{amsmath}
\usepackage[all,cmtip]{xy}
\usepackage{amsthm}
\usepackage{color}
\usepackage{enumitem}
\usepackage{tikz}
\usepackage{appendix}
\usepackage[colorlinks=true]{hyperref}

\usepackage{hyperref}
\hypersetup{breaklinks=true}
\input diagxy
\theoremstyle{plain}
  \newtheorem{thm}{Theorem}[section]
  \newtheorem{lem}[thm]{Lemma}
  \newtheorem{prop}[thm]{Proposition}
  \newtheorem{cor}[thm]{Corollary}
\theoremstyle{definition}
  \newtheorem{defn}[thm]{Definition}
  \newtheorem{exmp}[thm]{Example}
  \newtheorem{rem}[thm]{Remark}

  \newtheorem{prob}[thm]{Problem}

\DeclareMathAlphabet{\mathcal}{OMS}{cmsy}{m}{n}

\makeatletter
\def\ps@pprintTitle{%
 \let\@oddhead\@empty
 \let\@evenhead\@empty
 \def\@oddfoot{\centerline{\thepage}}%
 \let\@evenfoot\@oddfoot}
\makeatother

\newcommand{\ra}{\rightarrow}

\newcommand{\lra}{\longrightarrow}

\newcommand{\CL}{\mathcal{L}}

\newcommand{\CT}{\mathcal{T}}

\newcommand{\bbN}{\mathbb{N}}

\renewcommand{\leq}{\leqslant}
\renewcommand{\geq}{\geqslant}

\renewcommand{\phi}{\varphi}
\newcommand{\vep}{\varepsilon}
\newcommand{\tauTL}{\tau_{_{T,L}}}
\newcommand{\Delp}{\Delta^+}
\newcommand{\tauT}{\tau_{_{T}}}

\numberwithin{equation}{section}

\begin{document}

\begin{frontmatter}

%% Title, authors and addresses

%% use the tnoteref command within \title for footnotes;
%% use the tnotetext command for theassociated footnote;
%% use the fnref command within \author or \address for footnotes;
%% use the fntext command for theassociated footnote;
%% use the corref command within \author for corresponding author footnotes;
%% use the cortext command for theassociated footnote;
%% use the ead command for the email address,
%% and the form \ead[url] for the home page:
%% \title{Title\tnoteref{label1}}
%% \tnotetext[label1]{}
%% \author{Name\corref{cor1}\fnref{label2}}
%% \ead{email address}
%% \ead[url]{home page}
%% \fntext[label2]{}
%% \cortext[cor1]{}
%% \address{Address\fnref{label3}}
%% \fntext[label3]{}

\title{When is the operation $\tau_{_{T,L}}$ a triangle function on $\Delp$?}

%% use optional labels to link authors explicitly to addresses:
%% \author[label1,label2]{}
%% \address[label1]{}
%% \address[label2]{}

\author{Hongliang Lai}
\ead{hllai@scu.edu.cn}

\author{Mengyu Luo}\ead{lmysrem@qq.com}

\author{Jie Zhang \corref{cor}}
\ead{zjandzxz@qq.com}

\cortext[cor]{Corresponding author.}
\address{School of Mathematics, Sichuan University, Chengdu 610064, China}

\begin{abstract}
This paper resolves an open problem posed by Schweizer and Sklar in 1983. We establish that the binary operation $\tauTL$ is a triangle function on $\Delp$ if and only if the following three conditions hold: (a) $L$ is a continuous t-conorm on $[0, \infty]$ satisfying $(LCS)$; (b) $T$ is a t-norm on $[0, 1]$; and (c) $T$ is weakly left continuous, with left continuity required when $L$ is non-Archimedean.
\end{abstract}

\begin{keyword}
%% keywords here, in the form: keyword \sep keyword
   distance distribution function \sep triangle function \sep  triangular norm\sep triangular conorm 
%% PACS codes here, in the form: \PACS code \sep code

%% MSC codes here, in the form: \MSC code \sep code
%% or \MSC[2008] code \sep code (2000 is the default)
%\MSC[2020] 54A05 \sep 54E70 \sep 54E35
\end{keyword}

\end{frontmatter}

%\tableofcontents

%% \linenumbers

%% main text

\section{Introduction}
The notion of probabilistic metric spaces was introduced by Menger \cite{Menger1942} in 1942 as a generalization of classical metric spaces. In this framework, the distance between two points is described by a distribution function rather than a deterministic real number. To model the “addition” of distances in the interval $[0, \infty]$, \v{S}erstnev \cite{Serstnev1964} later introduced the concept of triangle functions, refining the triangle inequality originally proposed by Menger and Wald \cite{Wald1943}.

Motivated by \v{S}erstnev’s formulation, Schweizer and Sklar \cite{Schweizer1983} studied a family of operations denoted by $\tau_{T,L}$, constructed from a binary operation $T$ on $[0,1]$ and a binary operation $L$ on $[0,\infty]$. They established sufficient conditions under which $\tau_{T,L}$ becomes a (continuous) triangle function. However, by presenting a counterexample—the triangle function $\tau_{_D}$—they also demonstrated that these conditions are not necessary. This led them to pose the following open problem (see 7.9.6 in \cite{Schweizer1983}): characterize all pairs $(T, L)$ for which $\tau_{T,L}$ is a (continuous) triangle function.

In 1992, Ying \cite{Ying1992} derived a set of necessary conditions that closely approximate the known sufficient ones, yet the sufficiency of those conditions remained open. In this paper, we completely resolve this long-standing problem. We prove that $\tau_{T,L}$ is a triangle function on $\Delta^+$ if and only if the following three conditions are satisfied:
\begin{itemize}
\item[(a)] $L$ is a continuous t-conorm on $[0, \infty]$ satisfying the~(LCS) condition;
\item[(b)] $T$ is a t-norm on $[0, 1]$;
\item[(c)] $T$ is weakly left continuous, and left continuous whenever $L$ is non-Archimedean.
\end{itemize}

While the necessity of conditions (a)–(c) was already established by Ying \cite{Ying1992}, the sufficiency part—especially the associativity of $\tau_{T,L}$—presents a nontrivial difficulty. Specifically, when $L$ is Archimedean and $T$ is only weakly left continuous, the argument used in Theorem 7.2.4 of \cite{Schweizer1983} no longer applies. By introducing a subtle technical argument, we successfully establish the associativity of $\tau_{T,L}$ under the weakened continuity condition (see Lemma~\ref{Lem A} and Proposition~\ref{suff proo}), thereby completing the sufficiency proof.

The rest of this paper is organized as follows. In Section~2, we review basic concepts and results on t-norms, t-conorms, and triangle functions, mainly following \cite{Klement2000,Schweizer1983}, and give a precise statement of Schweizer and Sklar’s open problem. Section~3 is devoted to the proof of our main results.

\section{Preliminaries}
The terminology and notation used in this paper are primarily drawn from the literature \cite{Schweizer1983}.

\begin{defn}(\cite{Schweizer1983})
A \emph{distance distribution function}  is a map $\phi\colon [0,\infty]\to[0,1]$ such that
 $\phi(0)=0$, $\phi(\infty)=1$,
and $\phi$ is increasing and left-continuous on $]0,\infty[$. Let $\Delta^+$ denote the set of all distance distribution functions. The partial order of $\Delta^+$ is defined pointwisely.
\end{defn}

Among distance distribution functions,  the following will be needed frequently.
\begin{itemize}
\item For each $r<\infty$, define
$\vep_r(x)=\begin{cases}
            0,& 0\leq x\leq r;\\
            1,& r<x.
        \end{cases}
$
\item For $r=\infty$, define
        $\vep_\infty(x)=\begin{cases}
            0,& 0\leq x<\infty;\\
            1,& x=\infty.
        \end{cases}
$

\item For each $p\in[0,1]$, define 
$ V_p(x)=\begin{cases} 0, & x=0;\\
                        p, & 0<x<\infty;\\
                        1, & x=\infty.
            \end{cases}
$
\end{itemize}

The mappings $r\mapsto\vep_r$ and $p\mapsto V_p$ define embeddings into $\Delp$ with nice order-theoretic properties:
\begin{enumerate}
\item[(1)] The map $r\mapsto\vep_r$ sends the extended interval $[0,\infty]$ into $\Delp$. It converts suprema in $[0,\infty]$ into infima in $\Delp$, and infima in $[0,\infty]$ into suprema in $\Delp$. 
\item[(2)] The map $p\mapsto V_p$ sends the unit interval $[0,1]$ into $\Delp$. It preserves all suprema and infima.
\end{enumerate}
In particular, the endpoints satisfy $\vep_0=V_1$ and $\vep_\infty=V_0$.

\begin{defn}(\cite{Schweizer1983}) A \emph{triangle function} $\tau$ is a binary operation on $\Delta^+$ that is commutative, associative, increasing in each place and has $\vep_0$ as identity.
\end{defn}

\begin{defn}(\cite{Schweizer1983})
The class $\CT$ is the set of all binary operations $T$ on $[0,1]$ that is increasing in each place and  has $1$ as identity. A \emph{triangular norm} (briefly, a \emph{t-norm}) is   a binary operation in $\CT$ that is commutative and associative. 
\end{defn}
Let $T$ be a binary operation in $\CT$.  $T$ is \emph{continuous} if it is a continuous binary function on $[0,1]^2$.   $T$ is \emph{left continuous} if 
$$T(x,y)=\sup\{T(u,v)|u<x,v<y\}$$
for all $x, y\in]0,1]$. $T$  is \emph{weakly left continuous} if 
$$T(x,y)=\sup\{T(u,v)|u\leq x,v<y\text{ or } u<x, v\leq y\}$$
for all $x, y\in]0,1]$. 

There are some basic t-norms on the unit interval $[0,1]$ in the following.
\begin{itemize}
\item  Minimum: $M(x,y)=\min\{x,y\}$;
\item  Product: $\Pi(x,y)=x\cdot y$; 
\item  {\L}ukasiewicz t-norm: $W(x,y) =\max\{x+y-1, 0\}$;
\item  Nilpotent minimum:
${{nM}}(x,y)=\begin{cases}\min\{x,y\}, & 1<x+y,\\
                       0, &           x+y\leq 1;
          \end{cases}$
\item Drastic product: 
$D(x,y)=\begin{cases}\min\{x,y\}, & x=1\text{ or } y=1,\\
                     0, &          \text{otherwise}.
        \end{cases}$
\end{itemize}

The drastic product $D$ is the minimal t-norm and the minimum $M$ is the maximal t-norm. In fact, for any binary operation $T\in\CT$,  it always holds that
$$D\leq T\leq M.$$

The three t-norms $M$, $\Pi$ and $W$ are continuous. The t-norm $nM$ is left continuous but not continuous.  The  t-norm $D$ is weakly left continuous but not left continuous. By changing the values of the triangular norm $nM$ on its anti-diagonal, we obtain the following t-norm:  
$$\widehat{nM}(x,y)
=\begin{cases} \min\{x,y\}, & 1\leq x+y,\\
                         0, & x+y<1.
\end{cases}
$$
Then $\widehat{nM}$ is not weakly left continuous at the point $(0.5,0.5)$.

\begin{defn} (\cite{Schweizer1983})
The class $\CL$ is the set of all binary operations $L$ on $[0,\infty]$ that satisfy the conditions:
(L1) $L$ is onto, (L2) $L$ is increasing in each place, and (L3) $L$ is continuous except possibly at the points $(0,\infty)$ and $(\infty,0)$.
\end{defn}

\begin{prop} \label{L0 implies cont} If an operation $L\in\CL$ has $0$ as identity, then it is continuous. 
\end{prop}
\begin{proof}
Since $L$ is increasing in each variable and has $0$ as identity, we observe that
$$\lim_{x\ra 0\atop y\ra\infty}L(x,y)=\infty=L(0,\infty),\quad \lim_{x\ra\infty\atop y\ra 0}L(x,y)=\infty=L(\infty,0).$$
This shows that $L$ is continuous at the points $(0,\infty)$ and $(\infty,0)$. Since $L$ is already continuous at other points, we conclude that $L$ is continuous on the entire domain $[0,\infty]^2$. 
\end{proof}

 A binary operation $L\in\CL$ is said to be \emph{strictly increasing}(see \cite{Saminger2008,Schweizer1983}) if 
\[ u<u',v<v'\implies L(u,v)<L(u',v').\leqno{(LS)}\]
It is said to be \emph{conditionally strictly increasing}(see \cite{Ying1992}) if  
\begin{equation*}
u<u',v<v', L(u',v')<\infty \implies L(u,v)<L(u',v').\leqno{(LCS)}\end{equation*}

A \emph{triangular conorm} (shortly, a \emph{t-conorm}) on $[0,\infty]$ is an associative binary operation $L$ on $[0,\infty]$ that  is   commutative, increasing in each place, and has $0$ as identity. A t-conorm $L$ is   \emph{continuous} if $L$ is a continuous function with respect to the usual topology.

There is a one-to-one correspondence between t-norms on $[0,1]$ and t-conorms on $[0,\infty]$ by an order isomorphism $\phi: ([0,1],\leq)\lra ([0,\infty],\geq)$. For example, let $\phi(x)=-\ln x$ for all $x\in[0,1]$. Then for a t-norm $T$ on $[0,1]$, the corresponding t-conorm $T^*$ is given by 
$$T^*(x,y)=-\ln T(\mathrm{e}^{-x},\mathrm{e}^{-y}).$$
The corresponding t-conorms of  $M$, $\Pi$, $W$ and $D$ are listed as below:
\begin{itemize}
\item $M^*(x,y)=\max\{x,y\}$.
\item $\Pi^*(x,y)=x+y$.
\item $W^*(x,y)=-\ln(1-\min\{\mathrm{e}^{-x}+\mathrm{e}^{-y},1\})$. 
\item $D^*(x,y)=\begin{cases}\max\{x,y\},& x=0\text{ or } y=0;\\
                            \infty, & \text{otherwise}.
                            
\end{cases}$
\end{itemize}

Among all t-conorms on $[0,\infty]$, $M^*$ is the smallest and $D^*$ is the largest. That is, for any t-conorm $L$, it always holds that
$$M^*\leq L\leq D^*.$$
All the t-conorms $M^*$, $\Pi^*$ and $W^*$ satisfy the $(LCS)$ condition, but $W^*$ does not satisfy the $(LS)$ condition.  

An element $x$ in a set $X$ is called an idempotent element of a binary operation $S$ on $X$ if $S(x,x)=x$.
For a t-norm $T$, both $0$ and $1$ are idempotent elements, called the trivial idempotent elements. Similarly, $0$ and $\infty$ are called the trivial idempotent elements of a t-cornorm $L$ on $[0,\infty]$.

Clearly, all idempotent elements of a continuous t-norm (or t-conorm) form a closed subset of $[0,1]$ (or $[0,\infty]$). Thus, for any non-idempotent element $x\in]0,1[$ (or $x\in]0,\infty[$), among the idempotent elements below it, there exists a largest one, denoted as $a_x$; among the idempotent elements above it, there exists a smallest one, denoted as $b_x$. Restrict the binary operation $T$ on $[a_x,b_x]$, we obtain an ordered monoid, written as $([a_x,b_x],T)$. 

The following conclusion is of fundamental importance in the theory of continuous t-norms.

\begin{prop} {\rm (\cite{Alsina2006,Klement2000})} Let $T$ be a continuous t-norm on $[0, 1]$. If $x\in]0, 1[$ is non-idempotent, then the ordered monoid $([a_x,b_x],T)$ is isomorphic to either $([0,1],\Pi)$ or $([0,1],W)$.   In particular, if $T$ has no non-trivial idempotent elements, then $([0,1],T)$ is isomorphic to either $([0,1],\Pi)$ or $([0,1],W)$.
\end{prop}

Dually, for continuous t-conorms on $[0,\infty]$, we have the following: 
\begin{prop} Let $L$ be a continuous t-conorm on $[0, \infty]$. If $x\in]0, \infty[$ is non-idempotent, then    $([a_x,b_x],L)$ is isomorphic to either $([0,\infty],+)$ or $([0,\infty],W^*)$.   In particular, if $L$ has no non-trivial idempotent elements, then $([0,\infty],L)$ is either isomorphic to $([0,\infty],+)$ or $([0,\infty],W^*)$.
\end{prop}

A continuous t-conorm $L$ has no non-trivial idempotent elements if and only if it is \emph{Archimedean}. That is, for all $x,y\in]0,\infty[$ with $y<x$, there exists an $n\in\bbN$ such that the n-th power of $y$ under $L$, written as $y^n$, satisfies $x<y^n$.

 Let $x\oplus y=\min\{x+y,1\}$ for all $x,y\in[0,1]$. Then binary operation $\oplus$ is called the \emph{truncated addition}. The ordered monoid $([0,1],\oplus)$ is isomorphic to $([0,\infty],W^*)$. The isomorphism is given by the map $$\phi:[0,1]\lra[0,\infty], \phi(t)=-\ln(1-t).$$ 

\begin{cor}
A continuous Archimedean t-conorm $L$ on $[0,\infty]$ is isomorphic to either the addition $+$ on $[0,\infty]$  or  the truncated addition $\oplus$ on $[0,1]$. Thus,  each continuous Archimedean t-conorm $L$ on $[0,1]$ satisfies condition $(LCS)$ but it satisfies condition $(LS)$ if and only if it is isomorphic to the addition $+$.
\end{cor}

\begin{cor}
Let $L$ be a continuous t-conorm.  It satisfies the $(LS)$ condition if and only if for all non-idempotent element $x\in]0,\infty[$,  $([a_x,b_x],L)$ is isomorphic to $([0,\infty],+)$.
 It satisfies the $(LCS)$ condition if and only if for all non-idempotent element $x\in]0,\infty[$ with $b_x<\infty$,  $([a_x,b_x],L)$ is isomorphic to $([0,\infty],+)$.
\end{cor}

\begin{defn} (\cite{Schweizer1983}) 
For any $T\in\CT$ and $L\in\CL$,  $\tauTL$ is the map
 on $\Delp\times\Delp$ whose value, for any $f$, $g$ in $\Delp$, is the
function $\tauTL(f,g)$ defined on $[0,\infty]$
by
\[\tauTL(f,g)(x)=\sup_{L(u,v)=x}T(f(u),g(v)).\]
If $L=+$, then we drop the $L$ in $\tauTL$ and simply write $\tauT$.
\end{defn}

\begin{prop}\label{tauTL A} {\rm (\cite{Schweizer1983})}
For any $T\in\CT$ and $L\in\CL$, $\tauTL(f,g)$ is increasing on $[0,\infty]$ and satisfies $\tauTL(f,g)(\infty)=1$. Moreover, $\tauTL$ is increasing in each place on $\Delta^+ \times \Delta^+$.
\end{prop}

Generally, $\tauTL(f,g)$ need not be left continuous on $]0,\infty[$, hence not in $\Delp$. Consequently, $\tauTL$ is not necessarily a triangle function on $\Delp$. To ensure that $\tauTL$ becomes a triangle function, the operations $T\in\CT$ and  $L\in\CL$ must satisfy some additional conditions.

%It is well known that a continuous t-norm $T$ has no non-trivial idempotent elements if and only if it is Archimedean. This means that for all $x,y\in]0,1[$ with $x<y$, there exists an $n\in\bbN$ such that the n-th power of $y$ under $T$, written as $y^n$, satisfies $y^n< x$. 

%is a strictly increasing map from $[0,1]$ onto $[0,\infty]$ and $$\phi(x\oplus y)=W^*(\phi(x),\phi(y))$$ for all $x,y\in[0,1]$.  

The following theorem provides a set of sufficient conditions for $\tauTL$ to be a triangle function. By Proposition \ref{L0 implies cont}, it is clearly a reformulated version of Theorem 7.2.4 in \cite{Schweizer1983}.
\begin{thm}
Let $T$ be a left continuous t-norm on $[0,1]$ and $L$ be a continuous t-conorm on $[0,\infty]$ satisfying the $(LS)$ condition. Then $\tauTL$ is a triangle function.
\end{thm}

For $\tauTL$ to be a triangle function on $\Delp$, left-continuity of $T$ is sufficient but not necessary.

\begin{exmp}\cite{Schweizer1983}   The drastic product $D$ is not left continuous, however,  $\tau_{_{D}}$
is a triangle function on $\Delp$. In fact, for all $f,g\in\Delp$,  
$$\tau_{_{D}}(f,g)(x)=\max\{f(\max\{0,x-g^{\wedge}(1)\}), g(\max\{0,x-f^{\wedge}(1)\})\},$$
where $f^\wedge(1)=\sup\{t\in[0,1]|f(t)<1\}$.
Then one can easily see that $\tau_{_D}$ is a triangle function. 
\end{exmp}

Therefore, Schweizer and Sklar posed the following open problem, see Problem 7.9.6 in \cite{Schweizer1983}.
\begin{prob}\label{open prob}
 Find conditions on $T$ and $L$ that are both necessary and sufficient (rather than merely sufficient) for $\tauTL$ to be a triangle function. Similarly, find necessary and sufficient conditions on $T$ and $L$ for $\tauTL$ to be continuous.
\end{prob}

%\begin{lem} Let $L$ be a continuous t-conorm on $[0,\infty]$. If $x<L(u,v)<y$ for some $u,v,x,y\in [0,\infty]$, then there exist $a,a',b,b'\in[0,\infty]$ such that $a<u<a',b<v<b'$ and $L(a,b)=x,L(a',b')=y$.
%\end{lem}
%\begin{proof}  Define $f(t)=L(tu,tv)$ and $g(t)=L(u+t,v+t)$ for $t\in[0,1]$. By continuity and connectedness, $f$ and $g$ are surjective onto $[0,L(u,v)]$ and $[L(u,v),\infty]$, respectively. The result follows.\end{proof}

%The following proposition is a dual result of \cite[Theorem 2.1.9]{Alsina2006}.
%\begin{prop} If $L$ is a continuous Archimedean t-conorm, then for all $u<u'$ and $v$ in $[0,\infty]$, it holds that  $L(u,v')<\infty\implies L(u,v)<L(u',v)$.\end{prop}

\section{The main result}

In this section, we provide a complete answer to the open problem posed by Schweizer and Sklar, with the main result summarized in the following theorem.

\begin{thm}\label{main result} Given $T\in\CT$ and $L\in\CL$, the map $\tauTL$ is a triangle function on $\Delp$ if and only if the following conditions hold:
\begin{itemize}
\item[{\rm(a)}] $L$ is a continuous t-conorm and satisfies condition $(LCS)$;
\item[{\rm(b)}] $T$ is a t-norm on $[0,1]$;
\item[{\rm(c)}] $T$ is weakly left continuous and if $L$ is non-Archimedean then $T$ is left continuous.
\end{itemize}
\end{thm}

\begin{rem}
Ying has already established the necessity of the three conditions (a)-(c) in the literature \cite{Ying1992}. In fact, in that paper, item (3) in Section 3 and items (1)–(3) in Section 4 provide the proof for (a); items (4) and (5) in Section 4 give (b);  item (4) in Section (3) proves the first part of (c) and item (8) in Section 4 yields the second part of (c).
\end{rem} 

 Nevertheless, we choose to present a full proof here to enhance readability for the audience.  The following lemma can provide some technical assistance.

 \begin{lem} \label{tau eq}
Let $T\in\CT$ and $L\in\CL$. If $\tauTL$ is a triangle function, then for all $f,g\in\Delp$, and all $x\in]0,\infty[$,
$$\tauTL(f,g)(x)=\sup_{L(u,v)<x}T(f(u),g(v)).$$
\end{lem}
\begin{proof}
Since $h=:\tauTL(f,g)$ is left continuous at any $x\in]0,\infty[$, we have 
$$h(x)=\sup_{z<x}h(z)
=\sup_{z<x}\sup_{L(u,v)=z}T(f(u),g(v))=\sup_{L(u,v)<x}T(f(u),g(v)).$$
\end{proof}

Now we prove the necessity of the conditions (a)-(c). 
 
(a) \emph{$L$ is a continuous t-conorm and satisfies condition $(LCS)$.} 
\begin{proof} We divide the proof into two parts. 

(i) $L$ satisfies the $(LCS)$ condition. 
We aim to show that for any $u < u'$ and $v < v'$, if $L(u',v')<\infty$, then $L(u, v) < L(u', v')$. 
Let $x = L(u, v)$ and $x' = L(u', v') < \infty$. Define the function $h= \tauTL(\vep_u, \vep_v)$.   We analyze the values of $h$ at $x'$ and $x$.

First, consider $h(x')$. It holds that 
\[
h(x')=\sup_{L(r,s)=x'}T({\vep_u}(r),{\vep_v}(s))
\geq T({\vep_u}(u'),{\vep_v}(v'))
=1.
\]
Second, we compute $h(x)$. By Lemma \ref{tau eq}, it holds that:
\[
h(x)= \sup_{L(r, s) < x} T(\vep_u(r), \vep_v(s)).
\]
Now, suppose $L(r, s) < x$. We claim that either $r < u$ or $s < v$. Indeed, if both $r \geq u$ and $s \geq v$ held, then by the monotonicity of $L$, we would have $L(r, s) \geq L(u, v) = x$, contradicting $L(r, s) < x$. Thus, for any such pair $(r, s)$, we have $\vep_u(r) = 0$ or $\vep_v(s) = 0$, which implies $T(\vep_u(r), \vep_v(s)) = 0$. Consequently:
\[
h(x) = \sup_{L(r, s) < x} T(\vep_u(r), \vep_v(s)) = 0.
\]

Combining the results, we have $h(x) = 0 < 1 \leq h(x')$. 
Since $h$ is non-decreasing, it follows that $x < x'$, i.e., $L(u, v) < L(u', v')$.

(ii) $L$ is a continuous t-conorm on $[0,1]$. Our strategy is to show that for all $u,v\in[0,\infty]$, 
$$\tauTL({\vep_u},{\vep_v})=\vep_{L(u,v)}.$$ According to Lemma 7.2.13 in \cite{Schweizer1983}, the identity holds when $L$ satisfies the $(LS)$ condition. We show that it remains valid without the $(LS)$ condition on $L$. 
%We proceed by considering two cases based on the value of $L(u,v)$.

%\emph{Case 1}: $L(u,v)=\infty$. For any $x<\infty$, if $L(r,s)=x$, then $$L(r,s)<\infty=L(u,v).$$ By the monotonicity of $L$, this implies $r<u$ or $s<v$ (otherwise, $r\geq u$ and $s\geq v$ would force $L(r,s)\geq L(u,v)=\infty$, a contradiction). Hence, ${\vep_u}(r)=0$ or ${\vep_v}(s)=0$, and $$\tauTL({\vep_u},{\vep_v})(x)=\sup_{L(r,s)=x}T({\vep_u}(r),{\vep_v}(s))=0.$$ 
%Since $\vep_{L(u,v)}(x)=\vep_{\infty}(x)=0$ for $x<\infty$, the identity holds for all finite $x$. At $x=\infty$, we have $\tauTL({\vep_u},{\vep_v})(\infty)=1=\vep_\infty(\infty)$ by Proposition \ref{tauTL A}, completing the verification for this case.

Given $u,v\in[0,\infty[$, let $L(u,v)=x_0$ and $h=\tauTL(\vep_u,\vep_v)$. For any $x<x_0$, 
by Lemma \ref{tau eq}, 
$$h(x)=\sup_{L(r,s)<x}T(\vep_u(r),\vep_v(s)).$$
 For $r$ and $s$ with $L(r,s)<x$, we have that either $r<u$ or $s<v$. In fact, if both $r \geq u$ and $s \geq v$ held, then we would have $L(r, s) \geq L(u,v)=x_0$, contradicting $L(r, s) < x<x_0$. Thus, for any such pair $(r, s)$, we have $\vep_u(r) = 0$ or $\vep_v(s) = 0$. Consequently, it holds that for all $x<x_0$, 
\[
h(x) = \sup_{L(r, s) < x} T(\vep_u(r), \vep_v(s)) = 0.
\]
If $x_0=\infty$, then it holds that $$h=\vep_\infty=\vep_{x_0}.$$
If $x_0<\infty$, then
 $$h(x_0)=\sup_{x<x_0}h(x)=0.$$
 For any $x\in]x_0,\infty[$, the function $t\mapsto L(u+t,v+t)$ is continuous on $[0,\infty]$ and $x_0=L(u,v)<x<\infty=L(\infty,\infty)$. It ensures existence of $r>u$ and $s>v$ with $L(r,s)=x$. For such $r$ and $s$, ${\vep_u}(r)=1$ and ${\vep_v}(s)=1$, so 
 $$h(x)=\sup_{L(r,s)=x}T({\vep_u}(r),{\vep_v}(s))=1.$$
Therefore, we have shown that, for all $x\in[0,\infty]$, $h(x)=\vep_{x_0}(x)$.  So, it holds that, for all $u,v\in[0,\infty[$, $$\tauTL(\vep_u,\vep_v)=\vep_{L(u,v)}.$$
 
 Paricularly, for all $v\in[0,\infty[$, $\vep_v=\tauTL(\vep_0,\vep_v)=\vep_{L(0,v)}$, which implies that $L(v,0)=v$. Since $L$ is non-decreasing on each place, we have that $L(0,\infty)=\infty$, and $L(u,\infty)=\infty$ for all $u\in[0,\infty]$. Consequently, for all $u\in[0,\infty]$, 
 it holds that 
 $$\tauTL(\vep_u,\vep_\infty)=\vep_\infty
 =\vep_{L(u,\infty)}.$$
 A similar argument shows that 
 $$\tauTL(\vep_\infty,\vep_u)=\vep_\infty
 =\vep_{L(\infty,u)}.$$

 Therefore, we have shown that for all $u,v\in[0,\infty]$, it holds that
 $$\tauTL(\vep_u,\vep_v)=\vep_{L(u,v)}.$$
 Since $\tauTL$ is a triangle function, it is commutative and associative, and has $\vep_0$ as identity. Thus, $L$ is commutative and associative, and has $0$ as identity, that is, $L$ is a t-conorm. Moreover, by Proposition \ref{L0 implies cont}, $L$ is continuous. 
  \end{proof}

(b) \emph{\(T\) is a t-norm on \([0,1]\).}
\begin{proof}
It is show by Lemma 7.10 in \cite{Saminger2008}, but with slightly different preset conditions on \(T\) and \(L\). Notably, here we do not assume that \(L\) satisfies the \((LS)\) condition.

We aim to show that \(\tauTL(V_p, V_q) = V_{T(p,q)}\) for all \(p, q \in [0,1]\). Once this is established, the commutativity and associativity of \(T\) follow directly from the corresponding properties of \(\tauTL\).

Recall that \(V_p\) is defined by \(V_p(0) = 0\) and \(V_p(x) = p\) for all \(x\in]0,\infty[\).   We now verify the equality pointwise.
Notice that 
\[
\tauTL(V_p, V_q)(x) = \sup_{L(u,v) =x} T(V_p(u), V_q(v)).
\]
First, consider \(x = 0\). We know that \(L(u,v) = 0\) implies both \(u = 0\) and \(v = 0\) since $L$ is a t-conorm(it is shown by (a)).  Thus,  
 \[\tauTL(V_p, V_q)(0) = 0 = V_{T(p,q)}(0).\]
Second, consider any \(x \in ]0, \infty[\).  
Since $x\mapsto L(x,x)$ is a continuous function from $[0,\infty]$ onto $[0,\infty]$, there is some $u\in]0,\infty[$ such that $x=L(u,u)$. Clearly, we have $V_p(u)=p$ and $V_q(u)=q$. 
Moreover, for any pairs with \(L(u,v) = x\), we have $\max\{u,v\}\leq L(u,v)=x$ since $L$ is a t-conorm.  Thus, the value $V_p(u)$ is either $p$ or $0$,  and $V_q(v)$ is either $q$ or $0$. Thus, the supremum 
 \[\sup_{L(u,v)=x}T(V_p(u), V_q(v))
 =T(p,q)=V_{T(p,q)}(x).\]
Therefore, we conclude that \(\tauTL(V_p, V_q) = V_{T(p,q)}\). This completes the proof.
\end{proof}

 %(b) \emph{$T$ is a t-norm on $[0,1]$.}
%\begin{proof} It differs from Lemma 7.10 in \cite{Saminger2008} only slightly in the preset conditions on $T$ and $L$. The proof we present here does not require the assumption that $L$ satisfies the $(LS)$ condition.

 %We check that $$\tauTL(V_p,V_q)=V_{T(p,q)}$$ for all $p,q\in[0,1]$. Then commutativity and associativity of $T$ follow straightforwardly.

%Since $\tauTL$ is a triangle function on $\Delp$, we have $\vep_0=\tauTL(\vep_0,\vep_0)$. At the point $x=0$, it holds that  
%$\vep_0(0)=\sup_{L(u,v)=0}T(\vep_0(u),\vep_0(v))=0$. Thus, for any pair $(u,v)$ with $L(u,v)=0$, we have either $\vep_0(u)=0$ or $\vep_0(v)=0$, that is, either $u=0$ or $v=0$. So, we have that
%$$\tauTL(V_p,V_q)(0)
%=\sup_{L(u,v)=0}T(V_p(u),V_q(v))=0.$$
%At any point $x\in]0,\infty[$, it holds that 
%$\vep_0(x)=\sup_{L(u,v)=x}T(\vep_0(u),\vep_0(v))=1$.
%Thus, there is some $u,v$ with $L(u,v)=x$ such that 
%$\vep_0(u)=\vep_0(v)=1$. It implies that $u,v\in]0,\infty[$. So, we have that
%$$\tauTL(V_p,V_q)(x)
%=\sup_{L(u,v)=x}T(V_p(u),V_q(v))=T(p,q).$$
%Therefore, we have shown that for all $x\in[0,\infty]$, 
%\[\tauTL(V_p,V_q)(x)=V_{T(p,q)}(x).\]
%\end{proof}

(c) \emph{$T$ is weakly left continuous and if $L$ is non-Archimedean, then $T$ is left continuous.}

\begin{proof}
We aim to prove that $T$ is weakly left continuous, and furthermore, if $L$ is non-Archimedean, then $T$ is left continuous. 
It suffices to establish the following inequality for any given $x_0, y_0 \in ]0,1]$:
\[
T(x_0, y_0) \leq \sup\,\{ T(x, y) \mid x \leq x_0,\ y < y_0 \text{ or } x < x_0,\ y \leq y_0 \}.
\]

Since $L$ is continuous and $L(0,0)=0$, $L(\infty,\infty)=\infty$, there exists an $a \in ]0,\infty[$ such that $L(a, a) = z_0$ for any fixed $z_0 \in ]0,\infty[$. Define functions $f, g \in \Delta^+$ as follows:
\[
f(x) = 
\begin{cases}
\frac{x_0}{a}x, & 0 \leq x \leq a, \\
x_0,          & a < x < \infty, \\
1,            & x = \infty,
\end{cases}
\quad
g(x) = 
\begin{cases}
\frac{y_0}{a}x, & 0 \leq x \leq a, \\
y_0,          & a < x < \infty, \\
1,            & x = \infty.
\end{cases}
\]

For any $u, v \in [0, \infty]$ satisfying $L(u, v) = z_0$, the inequality $\max\{u, v\} \leq L(u, v) = z_0$ implies $u \leq z_0 < \infty$ and $v \leq z_0 < \infty$. Consequently, $f(u) \leq x_0$ and $g(v) \leq y_0$. We then observe:
\[
T(x_0, y_0) \geq \sup_{L(u,v)=z_0} T(f(u), g(v)) \geq T(f(a), g(a)) = T(x_0, y_0),
\]
which forces the equality:
\[
T(x_0, y_0) = \sup_{L(u,v)=z_0} T(f(u), g(v)) = \tauTL(f, g)(z_0).
\]

By Lemma \ref{tau eq}, at $z_0$, we have:
\[
\tauTL(f, g)(z_0)=\sup_{L(u,v) < z_0} T(f(u), g(v)).
\]

Now, for any $u, v \in [0, \infty]$ with $L(u, v) < z_0$, the condition $\max\{u, v\} \leq L(u, v) < z_0$ holds. If both $u \geq a$ and $v \geq a$ were true, then $L(u, v) \geq L(a, a) = z_0$ would follow, a contradiction. Hence, either $u < a$ or $v < a$. This implies either $f(u) < f(a) = x_0$ or $g(v) < g(a) = y_0$. Therefore,
\begin{align*}
T(x_0, y_0) &= \sup_{L(u,v) < z_0} T(f(u), g(v)) \\
            &\leq \sup\, \{ T(x, y) \mid x < x_0,\ y \leq y_0 \text{ or } x \leq x_0,\ y < y_0 \}.
\end{align*}
This establishes the weak left continuity of $T$.

To prove the second claim, suppose $L$ is non-Archimedean. Then there exists an idempotent element $z_0 \in ]0, \infty[$ such that $L(z_0, z_0) = z_0$. Take $a = z_0$ in the above construction. For any $u, v$ with $L(u, v) < z_0$, we have $\max\{u, v\} \leq L(u, v) < z_0 = a$, which implies $u < a$ and $v < a$. Consequently, $f(u) < x_0$ and $g(v) < y_0$. Thus,
\begin{align*}
T(x_0, y_0) &= \sup_{L(u,v) < z_0} T(f(u), g(v)) \\
            &\leq \sup\, \{ T(x, y) \mid x < x_0,\ y < y_0 \},
\end{align*}
proving that $T$ is left continuous.
\end{proof}

Next, we prove the sufficiency part of Theorem \ref{main result}. The following lemma is crucial, as it ensures that for an Archimedean t-conorm $L$ and a weakly left continuous t-norm $T$, $\tauTL$ satisfies the associative law.
  
\begin{lem}\label{Lem A}
Suppose $T$ is a t-norm on $[0,1]$ and $L$ is a continuous Archimedean t-conorm on $[0,\infty]$. Given $f,g\in\Delta^+$ and $y\in ]0,\infty[$, then for all $x>y$, there is a pair $(u_x,v_x)$ with $L(u_x,v_x)=x$ such that $T(f(u_x),g(v_x))\geq \tauTL(f,g)(y)$.
\end{lem}

\begin{proof} Since continuous Archimedean t-conorms $L$ on $[0,\infty]$ are isomorphic to either the addition $+$ or the t-conorm $W^*$, it suffices to prove it for addition $+$ and the t-conorm $W^*$. Here, we provide the proof only for addition $+$, as the proof for the t-conorm $W^*$ follows a similar argument.

 % We proceed by considering two distinct cases, based on whether the supremum defining $\tauTL(f,g)(y)$ is attained.

%\emph{Case 1}: The supremum is attained. Assume there exists a pair $(u_0,v_0)$ with $u_0+v_0=y$ such that $T(f(u_0),g(v_0))={\tauT}(f,g)(y)$. Given $x>y$,   we may fix $u_x=u_0$ and let $v_x=x-u_0$. Then $v_x>v_0$ and $u_x+v_x)=x$. Since $f$ and $g$ are non-decreasing, we have $f(u_x)=f(u_0)$ and $g(v_x)\geq g(v_0)$. By the monotonicity of $T$, it follows that $$T(f(u_x),g(v_x))\geq T(f(u_0),g(v_0))={\tauT}(f,g)(y),$$ which satisfies the requirement.

%\emph{Case 2}: The supremum is not attained. Suppose for all pairs $(u,v)$ with $u+v=y$, we have $T(f(u),g(v))<{\tauT}(f,g)(y)$. 
Since $$\tauT(f,g)(y)=\sup_{u+v=y}T(f(u),g(v)),$$
for each $n\in\bbN$, there exists a pair $(u_n,v_n)$ with $u_n+v_n=y$ such that
$$\tauT(f,g)(y)-\frac{1}{n}<T(f(u_n),g(v_n)).$$
%a sequence $\{(u_n,v_n)\}_{n\in\bbN}$ with $u_n+v_n=y$ for each $n$  such that the sequence $\{T(f(u_n),g(v_n))\}_{n\in \bbN}$ is strictly increasing and converges to ${\tauT}(f,g)(y)$. 
Since $\{(u,v)|u+v=y\}$ is bounded and closed in $[0,\infty]^2$, by compactness, there exists a convergent subsequence $\{u_{n_k},v_{n_k}\}_{k\in\bbN}$  with a limit $(u_0,v_0)$ and $u_0+v_0=y$.  Furthermore, we may assume, without loss of generality, that $u_{n_k}\leq u_{n_{k+1}}$ for all $k\in\bbN$ (the other case, $u_{n_k}\geq u_{n_{k+1}}$ for all $k\in\bbN$ is analogous).

Now for a fixed $x>y$, let $u_x=u_0$ and $v_x=x-u_0$.
Firstly, since $u_{n_k}$ is non-decreasing and $\lim\limits_{k\ra\infty}u_{n_k}=u_0=u_x$, we have $$f(u_{n_k})\leq f(u_x)$$ for all $k\in\bbN$.
Secondly,  because $v_{n_k}=y-u_{n_k}$ is non-increasing,  $\lim\limits_{k\ra\infty}v_{n_k}=v_0$, and $v_x=x-u_0>y-u_0=v_0$, there exists $K\in\bbN$ such that  $v_{n_k}<v_x$ for all $k>K$. It implies that $$g(v_{n_k})\leq g(v_x)$$ as $g$ is non-decreasing.  Therefore, by the monotonicity of $T$, $$T(f(u_x),g(v_x))\geq T(f(u_{n_k}),g(v_{n_k}))$$ for all $k>K$. Taking the supremum over $k>N$, we obtain $$T(f(u_x),g(v_x))\geq \sup_{k>K}T(f(u_{n_k}),g(v_{n_k}))={\tauT}(f,g)(y).$$ % Let $u_x=u$ and $v_x=L(v,\delta)$. Then $L(u_x,v_x)=L(u,L(v,\delta))=L(L(u,v),\delta)=L(y,\delta)=x$,  and the inequality holds. In both cases, we have constructed the required pair $(u_x,v_x)$, completing the proof.
\end{proof}

\begin{prop} \label{suff proo}
Given a continuous t-conorm $L$ on $[0,\infty]$ and a t-norm $T$ on $[0,1]$, if $L$ is Archimedean and $T$ is weakly left continuous, or $L$ is non-Archimedean and $T$ is left continuous, then $\tauTL$ is a triangle function on $\Delp$.
\end{prop}
\begin{proof} We divide the proof into two parts based on whether the t-conorm $L$ satisfies the Archimedean property.

\textbf{Part I}: Let $L$ be Archimedean and $T$  be weakly left continuous. Since continuous Archimedean t-conorms $L$ on $[0,\infty]$ are isomorphic to either the addition $+$ or the t-conorm $W^*$, it suffices to prove it for addition $+$ and the t-conorm $W^*$. Here, we provide the proof only for addition $+$, as the proof for the t-conorm $W^*$ follows a similar argument. 

We show that $\tauTL$ is a triangle function in four steps:

\emph{Step 1}: $\tauTL$ is a binary operation on $\Delp$. It is established by Ying in \cite{Ying1992} (   see item (2) in Section 3). Here we write the proof for convenience of readers.

It suffices to show that, for all $f,g\in\Delp$, ${\tauT}(f,g)$ is left continuous on $]0,\infty[$. 
 That is, at any point $x_0\in]0,\infty[$, for all $\vep>0$, there is $y<x_0$ such that 
    \[
     {\tauT}(f,g)(y)>{\tauT}(f,g)(x_0)-\epsilon.
    \]
    By the definition of ${\tauT}$, there is $u_0,v_0\in[0,\infty[$ with $u+v=x_0$ such that $$T(f(u_0),g(v_0))>{\tauT}(f,g)(x_0)-\frac{\epsilon}{2}.$$

    Since $T$ is weakly left continuous, there exist $a,b\in [0,1]$ with $a<f(u_0)$ and $b\leq g(v_0)$ (or symmetrically $a\leq f(u_0)$, and $b<g(v_0)$) such that
     $$T(a,b)>T(f(u_0),g(v_0))-\frac{\epsilon}{2}.$$

    If $a\leq f(u_0)$ and $b<g(v_0)$, then there exists $v'<v_0$ such that $b<g(v')$ since $g$ is left continuous at the point $v_0$. Let $y=u_0+v'$, then  $y<u_0+v_0=x_0$, and consequently, it holds that  
    \[
        {\tauT}(f,g)(y)\geq T(f(u_0),g(v')) \geq T(a,b) > T(f(u_0),g(v_0))-\frac{\epsilon}{2}>{\tauT}(f,g)(x_0)-\epsilon.
    \]
 
 If $a<f(u_0)$ and $b\leq g(v_0)$, then a symmetric argument yields the same inequality. In both cases, the desired inequality holds, proving left continuity.

\emph{Step 2}: ${\tauT}$ is commutative. For all $f,g\in\Delp$ and $x\in[0,\infty]$, the commutativity of $T$ and $L$ implies
\[
{\tauT}(f,g)(x)=\sup_{u+v=x}T(f(u),g(v)) =\sup_{v+u=x}T(g(v),f(u))={\tauT}(g,f)(x).
\]
Thus, ${\tauT}$ is commutative. 

\emph{Step 3}: $\vep_0$ is the identity of ${\tauT}$.
 For all $f\in\Delp$ and $x\in]0,\infty[$, 
 $${\tauT}(f,{\vep_0})(x)=\sup\limits_{u+v=x}T(f(u),{\vep_0}(v)).$$
 Since $\vep_0(0)$ and ${\vep_0}(v)=1$ if $v>0$, the supremum is attained when $v>0$. Thus, 
 $${\tauT}(f,{\vep_0})(x)=\sup\limits_{0<u<x}f(u)=f(x)$$   since $f$ is left continuous. Hence, ${\tauT}(f,{\vep_0})=f$.  

\emph{Step 4}: ${\tauT}$ is associative.
For all $f,g,h\in\Delta^+$, we show 
     \[
         {\tauT}({\tauT}(f,g),h)={\tauT}(f,{\tauT}(g,h)).
     \]
For any $x\in[0,\infty]$, the left side is 
\begin{align*}
{\tauT}({\tauT}(f,g),h)(x)
&=\sup_{u+v=x}T({\tauT}(f,g)(u),h(v))\\
&=\sup_{u+v=x}T(\sup_{r+s=u}T(f(r),g(s)),h(v)).
    \end{align*}
Now we check the equality
\[ \sup_{u+v=x}T(\sup_{r+s=u}T(f(r),g(s)),h(v))
=\sup_{a+b+c=x}T(T(f(a),g(b)),h(c)).
\]
%The values at $0$ and $\infty$ are trivial. '
On one hand, we show the $\geq$. For any $a,b,c\in[0,\infty]$ with $a+b+c=x$, let $u=a+b$ and $v=c$. Then $u+v=x$, and we have 
$$\sup_{r+s=u}T(f(r),g(s))\geq T(f(a),g(b)).$$
By the monotonicity of $T$,
$$T(\sup_{r+s=u}T(f(r),g(s)),h(v))\geq T(T(f(a),g(b)),h(c)).$$
Taking the supremum over all such $(a,b,c)$ yields $$\sup_{u+v=x}T(\sup_{r+s=u}T(f(r),g(s)),h(v))\geq
\sup_{a+b+c=x}T(T(f(a),g(b)),h(c)).$$
On the other hand, we show the $\leq$. We know that ${\tauT}({\tauT}(f,g),h)$ is left continuous by the step 1. Our strategy is to show that for all $y<x$,
     \[ \sup_{u+v=y}T(\sup_{r+s=u}T(f(r),g(s)),h(v))\leq\sup_{a+b+c=x}T(T(f(a),g(b)),h(c)).
     \]
     Indeed, if this holds, then by left continuity at $x$, the limit as $y\ra x^-$ of the left-hand side equals its value at $x$, giving the desired inequality.

     Now, fix $y<x$ and take any $u,v\in[0,\infty]$ with $u+v=y$. Let $w=x-v$, then $w>u$ and $w+v=x$. By Lemma \ref{Lem A}, there exists a pair $(r',s')$ with $r'+s'=w$ such that
      $$\sup_{r+s=u}T(f(r),g(s))={\tauT}(f,g)(u)
      \leq T(f(r'),g(s')).$$
      Then by the monotonicity of $T$, 
$$T\left(\sup_{r+s=u}T(f(r),g(s)),h(v)\right)\leq T\left(T(f(r'),g(s')),h(v)\right).$$
But since $r'+s'+v=w+v=x$, we have $$T(T(f(r'),g(s')),h(v))\leq\sup_{a+b+c=x}T(T(f(a),g(b)),h(c)).$$
Combining these inequalities, we obtain 
$$T\left(\sup_{r+s=u}T(f(r),g(s)),h(v)\right)
\leq \sup_{a+b+c=x}T(T(f(a),g(b)),h(c)).$$
Taking the supremum over all pairs $(u,v)$ with $u+v=y$ gives the desired inequality for each $y<x$.

Therefore, we conclude that
$${\tauT}({\tauT}(f,g),h)(x)=\sup_{a+b+c=x}T(T(f(a),g(b)),h(c)).$$
A similar argument shows that
$${\tauT}(f,{\tauT}(g,h))(x)=\sup_{a+b+c=x}T(f(a),T(g(b),h(c))).$$
By the associativity of $T$, the right-hand sides of these two expressions are equal. Therefore,
$${\tauT}({\tauT}(f,g),h)={\tauT}(f,{\tauT}(g,h)),$$
completing the proof of associativity.

%\begin{align*}
%T(\sup_{L(r,s)=u}T(f(r),g(s)),h(v))  &\leq  T(T(f(r'),g(s')),h(v))\\
%&=\sup_{L(L(a,b),c)=x}T(T(f(a),g(b)),h(c)).
%\end{align*}
% Therefore, we obtain
% \[\tauTL(\tauTL(f,g),h)(x)= \sup_{L(L(a,b),c)=x}T(T(f(a),g(b)),h(c))\]
      %for all $x\in[0,\infty]$. A similar calculation shows that
%\[ \tauTL(f,\tauTL(g,h))(x)=\sup_{L(a,L(b,c))=x}T(f(a),T(g(b),h(c))).\]
%Since $T$ and $L$ are associative, it is easy to see that
      %\[\sup_{L(L(a,b),c)=x}T(T(f(a),g(b)),h(c)) =\sup_{L(a,L(b,c))=x}T(f(a),T(g(b),h(c))).\]
     % In conclusion, we have
     % \[    \tauTL(\tauTL(f,g),h)=\tauTL(f,\tauTL(g,h)). \]

\textbf{Part II}: Let $L$ be non-Archimedean and $T$ be left continuous. As pointed out by Ying in \cite{Ying1992},  the proof of Theorem 7.2.4 in \cite{Schweizer1983} is also valid if the condition $(LS)$ is weakened as the condition $(LCS)$. Thus, $\tauTL$ is a triangle function on $\Delp$ if the t-conorm $L$ is Archimedean and the t-norm $T$ is left continuous.
 \end{proof}

Note that the mapping $p\mapsto V_p$ embeds the monoid $([0,1],T,1)$ into the monoid $(\Delp,\tauTL,\vep_0)$, and the mapping $x\mapsto\vep_x$ embeds the monoid $([0,\infty],L,0)$ into $(\Delp,\tauTL,\vep_0)$. The following corollary is therefore natural:

\begin{cor} {\rm (\cite{Ying1992})}
Given $T\in\CT$ and $L\in\CL$, the map $\tauTL$ is a continuous triangle function on $(\Delp,d_L)$ if and only if $T$ is a continuous t-norm on $[0,1]$ and $L$ is a continuous t-conorm on $[0,\infty]$ satisfying the condition $(LCS)$. 
\end{cor}

\section*{Acknowledgment}
The authors acknowledge the support of National Natural Science Foundation of China (12171342).

\end{document}